\documentclass{amsart}
\usepackage[mathscr]{eucal}
\usepackage{graphicx}
\usepackage{amscd}
\usepackage{amsmath}
\usepackage{amsthm}
\usepackage{amsxtra}
\usepackage{calc}
\usepackage{amsfonts}
\usepackage{amssymb}
\usepackage{latexsym}
\usepackage{pdfsync}
\usepackage{amsopn}
\usepackage{mathrsfs}
\usepackage{accents}
\usepackage[all]{xy}
\SelectTips{cm}{}
\calclayout

\usepackage[colorlinks=true,linkcolor=red,citecolor=green,bookmarks=true]{hyperref}

\newcommand{\rt}{\rightarrow}
\newcommand{\lra}{{\xymatrix@C=3em{\ar @{>} [r] & }}}
\newcommand{\lrt}{\longrightarrow}

\def\db{\operatorname{\mathsf{D^b}}}

\def\ds{\operatorname{\mathsf{D_{sg}}}}

\newcommand{\md}{\mathsf{mod}}

\newcommand{\st}{\stackrel}

\newcommand{\im}{{\mathsf{Im}}}

\newcommand{\MF} {\mathsf{MF}}

\newcommand{\G}{\mathcal{G} }
\newcommand{\umon}{\underline{\mathsf{Mon}}}

\newcommand{\n}{{\mathfrak{n}}}

\newcommand{\cok}{{\rm{Coker}}}
\newcommand{\Ker}{{\rm{Ker}}}

\newcommand{\Hom}{{\mathsf{Hom}}}

\newcommand{\id}{{\mathsf{id}}}

\newcommand{\mon}{{\mathsf{Mon}}}
\newcommand{\mor}{{\mathsf{Mor}}}

\newcommand{\Ext}{\mathsf{{Ext}}}

\newcommand{\C}{\mathcal{C}}
\newcommand{\cp} {\mathcal{P}}
\newcommand{\E}{\mathcal{E}}

\newtheorem{theorem}{Theorem}[section]

\newtheorem{cor}[theorem]{Corollary}

\newtheorem{lemma}[theorem]{Lemma}
\newtheorem{prop}[theorem]{Proposition}

\theoremstyle{definition}
\newtheorem{example}[theorem]{Example}

\newtheorem{remark}[theorem]{Remark}

\newtheorem{s}[theorem]{}

\theoremstyle{plain}

\theoremstyle{definition}
\newtheorem{dfn}[theorem]{Definition}

\numberwithin{equation}{section}

\begin{document}

\title[The monomorphism category of Gorenstein projective modules]
{The monomorphism category of Gorenstein projective modules and comparison with the category of matrix factorizations}

\author[Bahlekeh, Fotouhi, Nateghi and Salarian]{Abdolnaser Bahlekeh, Fahimeh Sadat Fotouhi, Armin Nateghi and Shokrollah Salarian}



\address{Department of Mathematics, Gonbad Kavous University, Postal Code:4971799151, Gonbad Kavous, Iran}
\email{bahlekeh@gonbad.ac.ir}
\address{School of Mathematics, Institute for Research in Fundamental Science (IPM), P.O.Box: 19395-5746, Tehran, Iran}
\email{ffotouhi@ipm.ir}

\address{Department of Pure Mathematics, Faculty of Mathematics and Statistics, University of Isfahan, P.O.Box: 81746-73441, Isfahan, Iran }
\email{a.nateghi92@sci.ui.ac.ir}


\address{Department of Pure Mathematics, Faculty of Mathematics and Statistics, University of Isfahan, P.O.Box: 81746-73441, Isfahan,
Iran and \\ School of Mathematics, Institute for Research in Fundamental Science (IPM), P.O.Box: 19395-5746, Tehran, Iran}
\email{Salarian@ipm.ir}


\subjclass[2020]{13D09, 18G80, 13H10, 13C60}

\keywords{monomorphism category, Gorenstein projective modules, matrix factorizations, Frobenius category, singularity category.}

\thanks{The research of the second author was in part supported by a grant from IPM}

\begin{abstract}
Let $(S, \n)$ be a commutative noetherian local ring and let $\omega\in\n$ be non-zero divisor. This paper is concerned with the category of monomorphisms between finitely generated Gorenstein projective $S$-modules, such that their cokernels are annihilated by $\omega$. We will observe that this category, which will be denoted by $\mon(\omega, \G)$, is an exact category in the sense of Quillen. More generally, it is proved that $\mon(\omega, \G)$  is a Frobenius category. Surprisingly, it is shown that not only the category of matrix factorizations embeds into $\mon(\omega, \G)$, but also its stable category as well as the singularity category of the factor ring $R=S/({\omega)}$, can be realized as triangulated subcategories of the stable category  $\underline{\mon}(\omega, \G)$.
\end{abstract}

\maketitle

\tableofcontents
\section{Introduction}
Assume that $(S, \n)$ is a commutative noetherian local ring, $\omega\in\n$ a non-zerodivisor and $R$ is the factor ring $S/{(\omega)}$.  A finitely generated  $S$-module $M$ is called Gorenstein projective, if it is reflexive, and  $\Ext^i_S(M, S)=0=\Ext^i_S(\Hom_S(M, S), S)$, for any $i\geq 1$. In this paper, we investigate the category of monomorphisms between finitely generated Gorenstein projective $S$-modules such that their cokernels are annihilated by $\omega$, which will be denoted by $\mon(\omega, \G)$.

The concept of finitely generated Gorenstein projective
modules over a commutative noetherian ring, which is a refinement of projective modules,   has been introduced by Auslander and Bridger in the mid-sixties \cite{auslander1969stable} under the name `modules of $G$-dimension zero'  to characterize Gorenstein local rings: a commutative noetherian local ring is Gorenstein if and only if its residue field has finite $G$-dimension. We denote the category of finitely generated projective (resp. Gorenstein projective) $S$-modules by $\cp(S)$ (resp. $\G(S)$). Also, the morphism category of finitely generated $S$-modules will be depicted by $\mor(S)$.  Evidently, $\mon(\omega, \G)$ is a full subcategory of $\mor(S)$. Indeed, $\mon(\omega, \G)$ is the category consisting of all monomorphisms $(G_1\st{f}\rt G_0)$, where $G_1, G_0\in\G(S)$ and $\omega\cok f=0$, and morphisms are commutative squares. The full subcategory of $\mon(\omega, \G)$ consisting of those objects $(P_1\st{f}\rt P_0)$ with $P_1,P_0\in\cp(S)$, will be denoted by $\mon(\omega, \cp)$. It is known that $\mor(S)$ is an abelian category, and so, it will be an exact category in the sense of Quillen  \cite{quillen2006higher}. {In this paper, we consider the category $\mon(\omega, \G)$ with those short exact sequences in $\mor(S)$ such that whose terms are in $\mon(\omega, \G)$. Our first goal is to show that, considering this class of short exact sequences as conflations, $\mon(\omega, \G)$ is an exact category, see Theorem \ref{sexact}.}  It should be noted that, according  to  Example \ref{nonex}, $\mon(\omega, \G)$ is not an extension-closed subcategory of $\mor(S)$, which forces us to verify the axioms of Definition \ref{exact}. More specifically, we show that $\mon(\omega, \G)$ is a Frobenius category, i.e., it is an exact category with enough projective objects and enough injective objects, and these classes of objects coincide. In particular, it is proved that its projective-injective objects are direct summands of finite direct sums of objects of the form $(P\st{\id}\rt P)\oplus(Q\st{\omega}\rt Q)$, for some projective $S$-modules $P, Q$, see Theorem \ref{frob}.

Our motivation to study the monomorphism category of finitely generated Gorenstein projective modules comes from the fact that there is a tie connection between $\mon(\omega, \G)$ and the category of matrix factorizations. On the other hand, the stable category of matrix factorizations as well as the singularity category of $R$, can be realized as triangulated subcategories of $\umon(\omega, \G)$.

The category of matrix factorizations of projective modules, denoted by $\mathsf{MF}(\omega, \cp)$,  is the category whose objects ordered pairs of $S$-homomorphisms ${\xymatrix{(P_1 \ar@<0.6ex>[r]^{\rho_1}& P_0\ar@<0.6ex>[l]^{\rho_0})}}$ in which $P_1, P_0\in\cp(S)$,  and the compositions $\rho_0\rho_1$ and $\rho_1\rho_0$ are the multiplications by $\omega$, moreover morphisms are chain maps. It is known that $\MF(\omega, \cp)$ is a Frobenius category, and in particular, its stable category is triangle equivalent to the singularity category of the factor ring $R=S/{(\omega)}$, provided $S$ is a regular ring, see \cite{buchweitz1987maximal} and  \cite[Theorem 3.9]{orlov2003triangulated}. For historical information on matrix factorizations, we refer the reader to Remark \ref{matrix}.

Inspired by the fact that the finitely generated Gorenstein projective modules are a generalization of projective modules, we define and study the so-called matrix factorization of Gorenstein projective $S$-modules,  by substituting the projective modules with
the Gorenstein projective modules. So, this category, which will be denoted by $\MF(\omega, \G)$,
can be viewed as a generalization of $\MF(\omega, \cp)$.  For a given object $(G_1\st{g_1}\rt G_0)$ in $\mon(\omega, \G)$, there is a unique morphism $(G_0\st{g_0}\rt G_1)$ such that $g_1g_0=\omega.\id_{G_0}$ and $g_0g_1=\omega.\id_{G_1}$, and so, we have the pair ${\xymatrix{(G_1 \ar@<0.6ex>[r]^{g_1}& G_0\ar@<0.6ex>[l]^{g_0})}}$ which is an object of $\MF(\omega, \G)$, see Lemma \ref{lem1}. The second purpose of this paper is to show that the functor sending each object  $(G_1\st{g_1}\rt G_0)\in\mon(\omega, \G)$ to the pair ${\xymatrix{(G_1 \ar@<0.6ex>[r]^{g_1}& G_0\ar@<0.6ex>[l]^{g_0})}}$ is an equivalence of categories, see Theorem \ref{sg}. This, in turn, yields that  $\MF(\omega, \G)$ is a Frobenius category, and particularly, its projective-injective objects are direct summands of finite direct sums of ${\xymatrix{(P \ar@<0.6ex>[r]^{\omega}& P\ar@<0.6ex>[l]^{\id})}}\oplus {\xymatrix{(Q \ar@<0.6ex>[r]^{\id}& Q\ar@<0.6ex>[l]^{\omega})}}$, for some projective $S$-modules $P, Q$, see Proposition \ref{gfrob1}. Since the aforementioned functor sends  projective objects in $\mon(\omega, \G)$ to projectives in $\MF(\omega, \G)$, there is an induced functor $\underline{\MF}(\omega, \G)\lrt\umon(\omega, \G)$, which is proved to be triangle equivalence, see Theorem \ref{22}. Restricting this to the projective modules, gives us a triangle equivalence functor $\umon(\omega, \cp)\lrt\underline{\MF}(\omega, \cp)$, see Corollary \ref{33}.

As the final result of this paper, we construct a fully faithful triangle functor $T:\underline{\G}(R)\lrt\umon(\omega, \G)$ which sends each object $M\in\G(R)$ to $(P\st{f}\rt G)$, arising from a short exact sequence of $S$-modules, $0\rt P\st{f}\rt G\rt M\rt 0$ with $P\in\cp(S)$ and $G\rt M$  a Gorenstein projective precover, see Theorem \ref{main}. This, combined with a fundamental result of Buchweitz and Happel \cite[Theorem 4.4.1]{buchweitz1987maximal}, gives rise to the existence of a fully faithful triangle functor from $\ds(R)$ to the stable category $\umon(\omega, \G)$, {provided that $R$ is Gorenstein},  see Corollary \ref{dd}. This result should be compared with a theorem of  Bergh and Jorgensen which asserts that there is a fully faithful triangle functor from that stable category $\underline{\MF}(\omega, \cp)$ to the singularity category $\ds(R)$, see  \cite[Theorem 3.5]{bergh2015complete}.


Throughout the paper, $(S, \n)$ is a commutative noetherian local ring, $\omega\in\n$ a non-zerodivisor, and $R$ stands for  the factor ring $S/{(\omega)}$. Also, all modules will be considered as finitely generated modules.

\section{Monomorphism category of Gorenstein projective modules}
 This section focuses on monomorphisms between Gorenstein projective $S$-modules in which their cokernels are annihilated by $\omega$, which will be denoted by $\mon(\omega, \G)$. We show that it is an exact category. Particularly, it is shown that $\mon(\omega, \G)$ is a Frobenius category, and {moreover, its projective-injective objects are completely recognized.}  Let us begin this by recalling the definition of finitely generated Gorenstein projective modules.

\begin{s}{\sc Finitely generated Gorenstein projective modules.} Assume that $M$ is a finitely generated $S$-module. Then, $M$ is said to be Gorenstein projective, if $\Ext^i_S(M, S)=0=\Ext^i_S(\Hom_S(M, S), S)$ for all $i\geq 1$ and the biduality map $\delta_M:M\lrt\Hom_S(\Hom_S(M, S), S)$, defined by $\delta_M(x)(g)=g(x)$ for any $g\in\Hom_S(M, S)$ and $x\in M$, is an isomorphism. The notion of finitely generated Gorenstein projective modules over a commutative noetherian ring has been introduced by Auslander and Bridger in \cite{auslander1969stable} under the name `modules of $G$-dimension zero'. Since then, this concept has found many applications in commutative algebra, algebraic geometry, singularity theory, and relative homological algebra. Enochs and Jenda \cite{enochs1995gorenstein} generalized this notion to any module over any ring, as the syzygy modules of totally acyclic complexes of projective modules.
\end{s}

\begin{dfn}
By the monomorphism category of Gorenstein projective modules, $\mon(\omega, \G)$, we mean a category that whose objects are those $S$-monomorphisms $(G_1\st{f}\rt G_0)$, where $G_1, G_0\in\G(S)$ and $\cok f$ is annihilated by $\omega$, that is to say, $\cok f$ is an $R$-module. Moreover, a morphism $\varphi=(\varphi_1, \varphi_0):(G_1\st{f}\rt G_0)\lrt (G'_1\st{f'}\rt G'_0)$ between two objects is a pair of morphisms $\varphi_1:G_1\rt G'_1$ and $\varphi_0:G_0\rt G'_0$ such that $f'\varphi_1=\varphi_0f$.

It is clear that $\mon(\omega, \G)$ is a full additive subcategory of the monomorphism category of finitely generated Gorenstein projective $S$-modules.
\end{dfn}

There has been a recent surge of interest in monomorphism categories, as they provide a framework for addressing open problems in linear algebra using tools and results from homological
algebra, combinatorics, and geometry, see for example \cite{ringel2008invariant, stable,ringel2014submodule,kussin2013nilpotent,xiong2014auslander,
asadollahi2022monomorphism,hafezi2024stable,hafezi2021subcategories,kosakowska2023abelian} and references therein.

In what follows, we show that $\mon(\omega, \G)$ is an exact category in the sense of Quillen. Recall that an exact category in the sense of Quillen is an additive category endowed with a class of kernel-cokernel pairs, called conflations, subject to certain axioms, see \cite[Definition 2.1]{buhler2010exact} and also \cite[Appendix A]{keller1990chain}. Exact categories generalize abelian categories and serve as a useful categorical framework for studying various subcategories of module categories.  For the reader's convenience, we give the precise definition.
{
\begin{dfn}\label{exact}(\cite[Definition 2.1]{buhler2010exact}) Assume that $\C$ is an additive category and $\E$  a class of kernel-cokernel pairs in $\C$. A kernel-cokernel pair $(i, p)$ in $\E$, which is also called an admissible pair, is a pair of composable morphisms $X'\st{i}\rt X\st{p}\rt X''$, where $i$ is a kernel of $p$ and $p$ is a cokernel of $i$. In this case, $i$ (resp. $p$) is called an admissible monic (resp. admissible epic). Admissible pairs, admissible monics, and admissible epics are also called conflations, inflations and deflations, respectively, see \cite{keller1990chain} and \cite{quillen2006higher}.\\ The pair $(\C, \E)$ is said to be an exact category (sometimes the class $\E$ is suppressed), if the following axioms hold:
\begin{itemize} \item[$(E0)$] For each object $C\in\C$, the identity morphism $\id_C$ is an admissible monic. \item [$(E0^{op})$] For each object $C\in\C$, the identity morphism $\id_C$ is an admissible epic.\item[$(E1)$] The class of admissible monics is closed under composition.\item[$(E1^{op})$] The class of admissible epics is closed under composition.\item[$(E2)$] The push-out of an admissible monic along an arbitrary morphism exists and yields an admissible monic.\item[$(E2^{op})$] The pull-back of an admissible epic along an arbitrary morphism exists and yields an admissible epic.
\end{itemize}

Abelian categories are naturally exact categories whose conflations are the class of all short exact sequences. Moreover, an extension-closed subcategory of an abelian category is an exact category in the same manner. Indeed, typical examples of exact categories arise in this way, see \cite[Lemma 10.20]{buhler2010exact}.
\end{dfn}

Since $S$ is a noetherian ring, the monomorphism category of $S$-modules, $\mathsf{Mon}(S)$, is an abelian category, and so, it will be an exact category. As $\mon(\omega, \G)$ is a full subcategory of $\mathsf{Mon}(S)$,
we will observe that this exact structure is inherited by $\mon(\omega, \G)$. {Namely, we show that, considering those short exact sequences in $\mathsf{Mon}(S)$ with terms in $\mon(\omega, \G)$ as conflations, $\mon(\omega, \G)$ is an exact category.}  According to the following example,  $\mon(\omega, \G)$ is not an extension-closed subcategory of $\mathsf{Mon}(S)$. Thus the canonical exact structure of $\mathsf{Mon}(S)$ does not transfer automatically to $\mon(\omega, \G)$. So, we need to check all the axioms of exact categories directly. This will be done through a series of results.}

\begin{example}\label{nonex}(\cite[Example 2.3]{stablebahlekeh})
Let $(S, \n)$ be a commutative Gorenstein local ring with $\mathsf{dim}S=1$ and let $\omega\in\n$ be non-zero divisor. Assume that $M$ is an $S$-module that is not annihilated by $\omega$ and $\omega^2M=0$, e.g. $M=S/{(\omega^2)}$. Set $N:=0:_{M}\omega$, the submodule consisting of those objects in $M$ which are annihilated by $\omega$, and $K:=M/N$. One should note that $K$ is annihilated by $\omega$, as well. Since $S$ is a one-dimensional Gorenstein ring, there exist short exact sequences of $S$-modules $0\rt G\st{f}\rt Q\st{\pi}\rt N\rt 0$ and $0\rt G'\st{f'}\rt Q'\st{\pi'}\rt K\rt 0$, where $G, G'\in\G(S)$ and $Q, Q'\in\cp(S)$, see for example \cite[Theorem 10.2.14]{enochs2011relative}. Thus $(G\st{f}\rt Q)$ and $(G'\st{f'}\rt Q')$ are objects of $\mon(\omega, \G)$. Now one may
apply the horseshoe lemma and get the short exact sequence $0\lrt (G\st{f}\rt Q)\lrt (G''\st{g}\rt Q\oplus Q')\lrt (G'\st{f'}\rt Q')\lrt 0$ in $\mathsf{Mor}(S)$. As $\cok g=M$ is not annihilated by $\omega$, the middle term does not belong to $\mon(\omega, \G)$. Consequently, $\mon(\omega, \G)$ is not an extension-closed subcategory of $\mathsf{Mor}(S)$ (and also, $\mon(S)$).
\end{example}

\begin{lemma}\label{e1}Let $(G_1\st{g_1}\rt G_0)\st{\varphi=(\varphi_1,\varphi_0)}\lrt(K_1\st{f_1}\rt K_0)$ and $(K_1\st{f_1}\rt K_0)\st{\theta=(\theta_1,\theta_0)}\lrt(T_1\st{h_1}\rt T_0)$ be two  {admissible monics} in $\mon(\omega, \G)$. Then $\theta\varphi$ is also an admissible monic.
\end{lemma}
\begin{proof}By our assumption, $\cok\varphi=(L_1\st{l_1}\rt L_0)$ and $\cok\theta=(E_1\st{e_1}\rt E_0)$ lie in $\mon(\omega, \G)$. We shall prove that $\cok\theta\varphi=(Z_1\st{z_1}\rt Z_0)\in\mon(\omega, \G)$, as well. {Consider the following commutative diagram of $S$-modules with exact rows and columns:

\[\xymatrix@C-0.5pc@R-.8pc{&&0\ar[d]&0 \ar[d]  && \\ 0\ar[r]&
G_i~\ar[r]^{\varphi_i}\ar@{=}[d]& K_i\ar[r]\ar[d]_{\theta_i}& L_i\ar[d]\ar[r] & 0 \\ 0 \ar[r]  & G_i~\ar[r]^{\theta_i\varphi_i} & T_i\ar[r]\ar[d]& Z_i\ar[d] \ar[r]& 0&\\
&& E_i\ar@{=}[r] \ar[d] & E_i
\ar[d]  &&\\ &&0& 0&&&\\ }\]where $i\in\{0, 1\}$}. In particular, one may obtain the commutative diagram with exact rows \[\xymatrix{0\ar[r]& L_1 \ar[r] \ar[d]_{l_1} & Z_1 \ar[r] \ar[d]_{z_1} & E_1 \ar[d]_{e_1}
\ar[r]&0\\0\ar[r]& L_0 \ar[r] & Z_0 \ar[r] & E_0\ar[r]&0. }\]As $l_1$ and $e_1$ are monomorphisms, the snake lemma yields that the same is true for $z_1$. Since $\G(S)$ is closed under extensions and $L_i, E_i\in\G(S),$ one infers that $Z_i\in\G(S)$. So, the proof will be completed, if we show that $\cok z_1\in\md R$.
To do this, take the following commutative diagram of $S$-modules with exact rows: \[\xymatrix{0\ar[r]& G_1 \ar[r]^{\theta_1\varphi_1} \ar[d]_{g_1} & T_1 \ar[r] \ar[d]_{h_1} & Z_1 \ar[d]_{z_1}\ar[r]&0
\\ 0\ar[r]&G_0 \ar[r]^{\theta_0\varphi_0} & T_0 \ar[r] & Z_0\ar[r]&0.}\] Since $z_1$ is a monomorphism, making use of the snake lemma, gives us the short exact sequence of $S$-modules
$0\rt\cok g_1\rt\cok h_1\rt\cok z_1\rt 0$. Now since $\cok h_1\in\md R$, one may deduce that the same is true for $\cok z_1$. So the proof is finished.
\end{proof}


The next proposition ensures the validity of axiom $(E2)$. In this direction, the following easy lemma plays a crucial role.
\begin{lemma}\label{ll}{Consider the following commutative diagram of $S$-modules with exact rows and columns:
\[\xymatrix@C-0.5pc@R-.8pc{&0\ar[d]&0 \ar[d] & && \\ 0\ar[r]&
M \ar[r]^{g} \ar[d]_{f} & X\ar[r]^{h}\ar[d]_{\psi} & L\ar[r]\ar@{=}[d] & \\ 0 \ar[r]  & N \ar[r]^{g'} \ar[d]_{\theta} &
Y \ar[d]_{\varphi} \ar[r]^{h'} & L  \ar[r]& 0&\\
& T\ar@{=}[r] \ar[d] & T
\ar[d] & &&\\ &0& 0& &&\\ }\]}If the upper row and left column are in $\md R$, then the same is true for the middle row and column.
\end{lemma}
\begin{proof}We only need to show that $Y$ is an $R$-module. To see this, take an arbitrary object $y\in Y$. So there is an object $n\in N$ such that $\theta(n)=\varphi(y)$. The commutativity of the bottom-most square yields that $g'(n)-y\in\im(\psi)$. Take an object $x\in X$ such that $g'(n)-y=\psi(x)$. By our hypothesis, $\omega n=0=\omega x$, implying that $\omega y=0$. This means that $Y$ is annihilated by $\omega$, and then, it will be an $R$-module, as desired.
\end{proof}

\begin{prop}\label{pushout}The push-out of any diagram {$$\begin{CD}(G_1\st{f}\rt G_0) @>{\varphi}>>(K_1\st{g}\rt K_0)\\@V\theta VV \\
(T_1\st{h}\rt T_0)\end{CD}$$} in $\mon(\omega, \G)$, where $\varphi=(\varphi_1,\varphi_0)$ is  {an admissible monic exists.}
\end{prop}
\begin{proof} Set $\cok\varphi:=(L_1\st{l}\rt L_0)$ and $\cok\theta:=(T'_1\st{h''}\rt T'_0)$. Since $\mon(S)$ is an exact category, because it is an abelian category, there is a push-out diagram
{
\[\xymatrix@C-0.5pc@R-.8pc{&0\ar[d]&0 \ar[d] & && \\ 0\ar[r]&
(G_1\st{f}\rt G_0) \ar[r]^{\varphi} \ar[d]_{\theta} & (K_1\st{g}\rt K_0)\ar[r]\ar[d]_{\psi} &  (L_1\st{l}\rt L_0)\ar[r]\ar@{=}[d] & \\ 0 \ar[r]  & (T_1\st{h}\rt T_0) \ar[r]^{\varphi'} \ar[d] &
(E_1\st{h'}\rt E_0) \ar[d] \ar[r] & (L_1\st{l}\rt L_0)  \ar[r]& 0&\\
& (T'_1\st{h''}\rt T'_0)\ar@{=}[r] \ar[d] & (T'_1\st{h''}\rt T'_0)
\ar[d] & &&\\ &0& 0& &&\\ }\]in $\mon(S)$, where rows and columns are exact. Since, by our assumption, $(L_1\st{l}\rt L_0)\in\mon(\omega, \G)$ and} $\G(S)$ is closed under extensions, one deduces that  $E_1, E_0\in\G(S)$. So, in order to complete the proof, we need to show that $\cok h'$ is an $R$-module. To see this, one may apply the functor cokernel to the above diagram and get the following {commutative diagram of $S$-modules with exact rows and columns:
\[\xymatrix@C-0.5pc@R-.8pc{&0\ar[d]&0 \ar[d] & && \\ 0\ar[r]&
\cok f~\ar[r]\ar[d]& \cok g\ar[r]\ar[d]& \cok l\ar@{=}[d]\ar[r] & 0\\ 0 \ar[r]  & \cok h~\ar[r]\ar[d] & \cok h'\ar[r]\ar[d]& \cok l  \ar[r]& 0&\\
&\cok h''~\ar@{=}[r]\ar[d]  & \cok h''
\ar[d] & &&\\ &0& 0& &&\\ }\]}By our hypothesis, the upper row and left column are in $\md R$. Thus, making use of Lemma \ref{ll} ensures that $\cok h'$ is an $R$-module. So the proof is finished.
\end{proof}

In light of these prerequisite, we can prove the result below, which says that the monomorphism category of Gorenstein projective $S$-modules is an exact category, where the conflations are those short exact sequences in $\mon(S)$ with terms in $\mon(\omega, \G)$.
\begin{theorem}\label{sexact} $\mon(\omega, \G)$ is an exact category.
\end{theorem}
\begin{proof}It is clear that the class of short exact sequences in $\mathsf{Mon}(S)$ with terms in $\mon(\omega, \G)$ is closed under isomorphisms. Moreover, axioms $(E0)$ and $(E0^{op})$ are obviously satisfied. The  axiom $(E1)$ is true, because of Lemma \ref{e1}. {Also, one may use the fact that $\G(S)$ is closed under kernels of epimorphisms, and apply the snake lemma, to seee that $\mon(\omega, \G)$ is colsed under kernels of epimorphisms. This, in particular, guarantees the validity of axiom $(E1^{op})$, as well.} Furthermore, axiom $(E2)$ is satisfied, thanks to Proposition \ref{pushout}. Finally, {the validity of axiom $(E2^{op})$ is obtained by dualizing the argument used in the proof of    Proposition \ref{pushout}}. So the proof is finished.
\end{proof}

It should be noted that, since $\cp(S)$ is closed under extensions, $\mon(\omega, \cp)$ will be an extension-closed subcategory of $\mon(\omega, \G)$. This fact, combined with \cite[Lemma 10.20]{buhler2010exact}, immediately leads to the result below.

\begin{cor}$\mon(\omega, \cp)$ is an extension-closed exact subcategory of $\mon(\omega, \G)$.
\end{cor}

The remainder of this section is devoted to show that $\mon(\omega, \G)$ is a Frobenius category. Recall that an exact category is called a Frobenius category if it has enough projectives and injectives and the projectives coincide with the injectives. The importance of Frobenius categories lies in their natural relationship with triangulated categories. Namely, the stable category of a Frobenius category admits a triangulated structure, where the shift functor is given by the inverse of the syzygy functor on the stable category, see \cite{happel1988triangulated, keller1990chain}. Triangulated categories of this type, are referred to as “algebraic” in the terminology of Keller  \cite{keller2006differential}.

\begin{lemma}\label{lem3}If $P$ is a  projective $S$-module, then
$(P\st{\id}\rt P)$ and $(P\st{\omega}\rt P)$ are
projective objects in $\mon(\omega, \G)$.
\end{lemma}
\begin{proof}
First, we prove the case for $(P\st{\id}\rt P)$. To do this,
assume that $0\lrt(E_1\st{e_1}\rt E_0)\lrt(G_1\st{g_1}\rt G_0)\st{\varphi=(\varphi_1,\varphi_0)}\lrt(P\st{\id}\rt P)\lrt 0$
is a short exact sequence in $\mon(\omega, \G)$. Since $0\rt E_1\rt G_1\st{\varphi_1}\rt P\rt 0$ is a short exact sequence of $S$-modules and $P$ is projective, there exists a morphism $\psi_1:P\rt G_1$ such that $\varphi_1\psi_1=\id_P$.
Set $\psi_0:=g_1\psi_1$. Now it is easily seen that $\varphi\psi=\id_{(P\st{id}\rt P)}$, and so $(P\st{\id}\rt P)$ is a projective object of $\mon(\omega, \G)$. Next we show that $(P\st{\omega}\rt P)$ is also a projective object of $\mon(\omega, \G)$. Take a short exact sequence $0\lrt (E_1\st{e_1}\rt E_0)\lrt (G_1\st{g_1}\rt G_0)\st{\varphi=(\varphi_1,\varphi_0)}\lrt (P\st{\omega}\rt P)\lrt 0$ in $\mon(\omega, \G)$. As $P$ is projective, there is a morphism $\psi_0:P\rt G_0$ such that $\varphi_0\psi_0=\id_P$. Since $\omega\cok g_1=0$, there is a morphism $\psi_1:P\rt G_1$ making the following diagram commutative
{\footnotesize\[\xymatrix{ & P\ar[d]^{\psi_0\omega}\ar[dl]_{\psi_1} & \\ G_1\ar[r]^{g_1}~& G_0\ar[r]& \cok g_1 .}\]}Now using the fact that $\omega$ is non-zerodivisor, one  may easily infer that $\varphi\psi=\id_{(P\st{\omega}\rt P)}$, and so $(P\st{\omega}\rt P)$ is a projective object of $\mon(\omega, \G)$. Thus the proof is finished.
\end{proof}

\begin{lemma}\label{lem4}
Let $P$ be a projective $S$-module. Then $(P\st{\id}\rt P)$ and $(P\st{\omega}\rt P)$ are
injective objects in $\mon(\omega, \G)$.
\end{lemma}
\begin{proof}
Let us prove only the case for $(P\st{\omega}\rt P)$. The other one is obtained easily. Assume that $0\lrt (P\st{\omega}\rt P)\st{\varphi=(\varphi_1,\varphi_0)}\lrt (E_1\st{e_1}\rt E_0)\lrt (G_1\st{g_1}\rt G_0)\lrt 0$ is a short exact sequence in $\mon(\omega, \G)$. Since $P$ is an injective object in $\G(S)$, there is a morphism $\psi_1:E_1\rt P$ such that $\psi_1\varphi_1=\id_P$.
Consider the following push-out diagram:
\[\xymatrix{&\\ \gamma:0 \ar[r] & E_1 \ar[r]^{e_1} \ar[d]_{\omega\psi_1} & E_0 \ar[r] \ar[d] & \cok e_1 \ar[r] \ar@{=}[d] & 0
\\ (\omega\psi_1)\gamma: 0 \ar[r] & P \ar[r] & T \ar[r] & \cok e_1 \ar[r] & 0. &}\]
Since $\omega\cok e_1=0$, applying \cite[Lemma 2.11]{stablebahlekeh} yields that the lower row is split. Hence, there is a morphism $\psi_0:E_0\lrt P$ such that $\psi_0e_1=\omega \psi_1$. Now it can be easily checked that $\psi\varphi=\id_{(P\st{\omega}\rt P)}$, and so, the proof is completed.
\end{proof}

In the next couple of results, we show that $\mon(\omega, \G)$ has enough projectives and injectives. In this direction, the following lemma plays a key role.
\begin{lemma}\label{lem1}
Let $(G_1\st{g_1}\rt G_0)$ be an arbitrary object of $\mon(\omega, \G)$. Then there exists
an  $S$-homomorphism $(G_0\st{g_0}\rt G_1)$ such that $g_1g_0=\omega.\id_{G_0}$ and $g_0g_1=\omega.\id_{G_1}$. In particular, $(G_0\st{g_0}\rt G_1)$ is an object of $\mon(\omega, \G)$.
\end{lemma}
\begin{proof}
By our hypothesis, there exists a short exact sequence of $S$-modules $0\rt G_1\st{g_1}\rt G_0\rt\cok g_1\rt 0$ such that $\cok g_1$ is an $R$-module. Consider the following push-out diagram:
\[\xymatrix{&\\ \eta:0 \ar[r] & G_1 \ar[r]^{g_1} \ar[d]_{\omega} & G_0 \ar[r] \ar[d] & \cok g_1 \ar[r] \ar@{=}[d] & 0
\\ \omega\eta: 0 \ar[r] & G_1 \ar[r] & T \ar[r] & \cok g_1 \ar[r] & 0.}\]
Since $\omega\cok g_1=0$, $\omega\eta$ will be split. Thus, one may find an $S$-homomorphism $g_0:G_0\rt G_1$ such that $g_0g_1=\omega.\id_{G_{1}}$. Another use of the fact that $\omega\cok g_1=0$, leads us to infer that $\omega G_0\subseteq g_1(G_1)$. This fact in conjunction with equality $g_0g_1=\omega.\id_{G_1}$ gives rise to the equality $g_1g_0=\omega.\id_{G_0}$. Now we show that $(G_0\st{g_0}\rt G_1)$ is an object of $\mon(\omega, \G )$. As $g_0$ is evidently a monomorphism, we only need to prove that $\cok g_0$ is an $R$-module. Consider the short exact sequence of $S$-modules $0\rt G_0\st{g_0}\rt G_1\st{\pi}\rt\cok g_0\rt 0$. Since $g_0g_1=\omega.\id_{G_1}$, $\omega G_1\subseteq g_0(G_0)$. Thus for any $y\in\cok g_0$, there exists an object $x\in G_1$ such that $\pi(x)=y$. So we have that $\omega y=\omega\pi(x)=\pi(\omega x)\in\pi g_0(G_0)=0$. This indeed  means that $\omega\cok g_0=0$, and so, $\cok g_0$ is an $R$-module, as required.
\end{proof}

\begin{lemma}\label{lem2}
The category $\mon(\omega, \G)$ has enough projectives.
\end{lemma}
\begin{proof}
Assume that $(G_1\st{g_1}\rt G_0)$ is an arbitrary object of $\mon(\omega, \G)$. Take
epimorphisms $h_1:P_1\rt G_1$ and $h_0:P_0\rt G_0$, where $P_1, P_0\in\mathcal{P}(S)$.
Consider the object $(P_1\oplus P_0\st{\id\oplus\omega}\lrt P_1\oplus P_0)$ in $\mon(\omega, \G)$ which
is projective, thanks to Lemma \ref{lem3}. Since $(G_1\st{g_1}\rt G_0)\in\mon(\omega, \G)$,  Lemma \ref{lem1} gives us a unique  monomorphism $g_0:G_0\rt G_1$ with $\cok g_0\in\md R$. So, we may have an epimorphism $(P_1\oplus P_0\st{\id\oplus\omega}\lrt P_1\oplus P_0)\st{\varphi=(\varphi_1,\varphi_0)}\lrt (G_1\st{g_1}\rt G_0)$ in $\mon(\omega, \G)$, where $\varphi_0=[g_1h_1~~h_0]$ and $\varphi_1=[h_1~~g_0h_0]$. So, as indicated in the proof of Theorem \ref{sexact}, $\Ker\varphi=(E_1\st{e_1}\rt E_0)\in\mon(\omega, \G)$, as well. Thus the proof is finished.
\end{proof}

\begin{lemma}\label{lem22}
 $\mon(\omega, \G)$ has enough injectives.
\end{lemma}
\begin{proof}
Consider an arbitrary object $(G_1\st{g_1}\rt G_0)\in\mon(\omega, \G)$. Since $G_1, G_0\in\G(S)$, one may take short exact sequences of $S$-modules, $0\rt G_1\st{h_1}\rt Q_1\st{f_1}\rt G'_1\rt 0$ and $0\rt G_0\st{h_0}\rt Q_0\st{f_0}\rt G'_0\rt 0$, where $Q_1, Q_0\in\mathcal{P}(S)$ and $G'_1, G_0'\in\G(S)$. By Lemma \ref{lem1}, there is an object $(G_0\st{g_0}\rt G_1)\in\mon(\omega, \G)$.
So we will obtain a monomorphism $(G_1\st{g_1}\rt G_0)\st{\varphi}\lrt(Q_1\oplus Q_0\st{\omega\oplus\id}\rt Q_1\oplus Q_0)$, where $\varphi_1=[h_1~~ h_0g_1]^t$ and $\varphi_0=[h_1g_0~~h_0]^t$. In view of Lemma \ref{lem4}, $(Q_1\oplus Q_0\st{\omega\oplus\id}\rt Q_1\oplus Q_0)$ is an injective object of $\mon(\omega, \G)$. So, in order to complete the proof, it suffices to show that  $\cok\varphi=(\varphi_1,\varphi_0)$ is an object of $\mon(\omega, \G)$. This will be done by showing that there exists a short exact sequence, $0\lrt (G_1\st{g_1}\rt G_0)\st{\varphi}\lrt(Q_1\oplus Q_0\st{\omega\oplus\id}\rt Q_1\oplus Q_0)\st{\psi}\lrt(G'_1\oplus Q_0\st{l}\rt Q_1\oplus G'_0)\lrt 0$ in $\mon(\omega, \G)$.\\
$\bullet$ Since $Q_1, Q_0\in\mathcal{P}(S)$ and $G'_1, G'_0\in\G(S)$, there exist morphisms $\alpha:Q_1\rt Q_0$ and $\beta:Q_0\rt Q_1$ such that $\alpha h_1=h_0g_1$ and $\beta h_0=h_1g_0$. Thus we obtain induced morphisms $\varepsilon:G'_1\rt G'_0$  and $\varepsilon':G'_0\rt G'_1$  such that $\varepsilon f_1=f_0\alpha$ and  $\varepsilon'f_0=f_1\beta$.\\
$\bullet$ According to the above equalities, we have $(\omega - \beta\alpha)h_1=0$, and so, there exists
a morphism $\gamma: G'_1\rt Q_1$ such that $\gamma f_1=\omega - \beta\alpha$. {Analogously, one gets a morphism  $\gamma':G'_0\rt Q_0$ such that $\gamma'f_0=\omega-\alpha\beta$.}
Now, by letting $l_1:=\tiny {\left[\begin{array}{ll} \gamma & -\beta \\ {\varepsilon} & {f_0} \end{array} \right],}$
one may deduce that $(G'_1\oplus Q_0\st{l_1}\rt Q_1\oplus G'_0)$ is an object of $\mon(\omega, \G)$. {Indeed, considering the morphism $(Q_1\oplus G'_0\st{l_0}\rt G'_1\oplus Q_0)$,  where $l_0=\tiny {\left[\begin{array}{ll} f_1 & \varepsilon' \\ -\alpha & \gamma' \end{array} \right],}$ we have $l_0l_1=\tiny {\left[\begin{array}{ll} \omega & 0 \\ 0 & \omega \end{array} \right]}=\omega\oplus\omega$. This forces $l_1$ to be a monomorphism.} In particular, it is easy to check that $0\lrt (G_1\st{g_1}\rt G_0)\st{\varphi}\lrt(Q_1\oplus Q_0\st{\omega\oplus\id}\rt Q_1\oplus Q_0)\st{\psi}\lrt(G'_1\oplus Q_0\st{l_1}\rt Q_1\oplus G'_0)\lrt 0$, with $\psi_1=\tiny {\left[\begin{array}{ll} f_1 & 0 \\ {-\alpha} & {\id} \end{array} \right]}$ and $\psi_0=\tiny {\left[\begin{array}{ll} \id & -\beta \\ {0} & {f_0}\end{array} \right],}$ is a short exact sequence in { $\mon(S)$.  Since the middle term of this sequence lies in $\mon(\omega, \G)$, applying the snake lemma yields that $\cok l_1$ is an $R$-module. Thus the latter sequence is indeed a short exact sequence in $\mon(\omega, \G)$. }So we are done.
\end{proof}

Now we are ready to state the following result.

\begin{theorem}\label{frob}
$\mon(\omega, \G)$ is a Frobenius category. In particular, its projective-injective objects are  equal to direct summands of finite direct sums of objects of the form $(P\st{\id}\rt P)\oplus(Q\st{\omega}\rt Q)$ for some $P, Q\in\cp(S)$.
\end{theorem}
\begin{proof}In view of Theorem \ref{sexact}, $\mon(\omega, \G)$ is an exact category. Moreover, by Lemmas \ref{lem2} and \ref{lem22}, $\mon(\omega, \G)$ has enough projective and injective objects. Furthermore, the coincidence of projective and injective objects will follow from the second assertion. So let us prove the second assertion. To do this, assume that $(G_1\st{g_1}\rt G_0)$ is a projective object in $\mon(\omega, \G)$. The proof of Lemma \ref{lem2}, gives rise to the existence of a short exact sequence,
$0\lrt (L_1\rt L_0)\lrt(Q_1\oplus Q_0\st{\id\oplus\omega}\lrt Q_1\oplus Q_0)\st{\psi}\lrt (G_1\st{g_1}\rt G_0)\lrt 0$ in $\mon(\omega, \G)$, where $Q_0, Q_1\in\mathcal{P}(S)$. Now the projectivity of $(G_1\st{g_1}\rt G_0)$ forces  $\psi$ to be a split epimorphism. Thus, $(G_1\st{g_1}\rt G_0)$ will be a direct summand of $(Q_1\st{\id}\rt Q_1)\oplus (Q_0\st{\omega}\rt Q_0)$.  The case for injective objects is obtained dually. So the proof is finished.
\end{proof}

It is known that $\cp(S)$ is closed under kernels of epimorphisms. Moreover, any Gorenstein projective module with finite projective dimension is projective. These facts ensure that $\mon(\omega, \cp)$ is an admissible subcategory of $\mon(\omega, \G)$. Recall that an extension-closed exact subcategory $\mathcal{F'}$ of a Frobenius category $\mathcal{F}$ is said to be
an {\em admissible subcategory}, if any object $X\in\mathcal{F'}$  there exist conflations  $X\rt Q\rt X'$ and $X''\rt P\rt X$ in $\mathcal{F'}$ such that $P, Q$ are projective objects of $\mathcal{F}$. In this case, $\mathcal{F'}$ will be also a Frobenius category, and in particular, an object $X\in\mathcal{F'}$ is projective if and only if it is projective as an object of $\mathcal{F}$,  see \cite[page 46]{chen2012three}. So, we will derive the result below.

\begin{cor}\label{s1}$\mon(\omega, \cp)$ is a Frobenius subcategory of $\mon(\omega, \G)$, with the same projective-injective objects.
\end{cor}

\section{Comparison with the categories of matrix factorizations and Gorenstein projective modules} In this section first, we shall define the notion of matrix factorizations of Gorenstein projective modules, which generalizes the classical notion of matrix factorizations. It is proved that the category of matrix factorizations of Gorenstein projective modules is equivalent to the monomorphism category of Gorenstein projective modules. In particular, it will be observed that their stable categories are triangle equivalent. It is also shown that, there exists a fully faithful triangle functor from the stable category of Gorenstein projective $R$-modules, $\underline{\G}(R)$,  to $\umon(\omega, \G)$. Let us inaugurate this section by recalling the definition of matrix factorizations.

\begin{s}\label{matrix}{\sc Matrix factorizations.} Recall that the category of matrix factorizations of projective modules, $\MF(\omega, \cp)$,  is the category whose objects are ordered pairs ${\xymatrix{(P_1 \ar@<0.6ex>[r]^{\rho_1}& P_0\ar@<0.6ex>[l]^{\rho_0})}}$ in which $P_1, P_0\in\cp(S)$, $\rho_1, \rho_0$ are $S$-homomorphisms, and the compositions $\rho_0\rho_1$ and $\rho_1\rho_0$ are the multiplications by $\omega$. Also, a morphism $\Psi=(\psi_1, \psi_0) :{\xymatrix{(P_1 \ar@<0.6ex>[r]^{\rho_1}& P_0\ar@<0.6ex>[l]^{\rho_0})}}\rt {\xymatrix{(Q_1 \ar@<0.6ex>[r]^{q_1}& Q_0\ar@<0.6ex>[l]^{q_0})}}$ in $\MF(\omega, \cp)$ is a pair of $S$-homomorphisms $\psi_1:P_1\rt Q_1$ and $\psi_0:P_0\rt Q_0$ such that $\psi_1\rho_0=q_0\psi_0$ and $q_1\psi_1=\psi_0\rho_1$. It should be remarked that $\MF(\omega, \cp)$ was originally denoted by $\MF(S, \omega)$ and called the category of matrix factorizations of $\omega$.

It is known that $\MF(\omega, \cp)$ is a Frobenius category, and in particular, its projective-injective objects
are equal to direct summands of finite direct sums of objects of the form ${\xymatrix{(P \ar@<0.6ex>[r]^{\omega}& P\ar@<0.6ex>[l]^{\id})}}\oplus {\xymatrix{(Q \ar@<0.6ex>[r]^{\id}& Q\ar@<0.6ex>[l]^{\omega})}}$, for some $P, Q\in\cp(S)$. So its stable category $\underline{\MF}(\omega, \cp)$ admits triangulated structure, see for example \cite{yoshino1990maximal}.
\end{s}
\begin{remark}\label{matrix}Matrix factorizations were introduced by Eisenbud in \cite{eisenbud1980homological} as a tool for compactly describing the minimal free resolutions of stable maximal Cohen-Macaulay modules over a local hypersurface ring. Since then, they have been used extensively in singularity theory. A famous result of  Buchweitz \cite{buchweitz1987maximal} indicates that, over a regular ring, the stable category of matrix factorizations, which is the same as the homotopy category of matrix factorizations, is triangle equivalent to the singularity category of the corresponding factor ring.

In recent years, matrix factorizations { have gained a wider audience due to their rediscovery in homological mirror symmetry \cite{katzarkov2008hodge},  knot theory \cite{khovanov2004matrix}, and string theory \cite{kapustin2004d, orlov2003triangulated}}.   More precisely, matrix factorizations were proposed by Kontsevich as descriptions of B-branes in Landau-Ginzburg models in topological string theory. As such they appear in the framework of mirror symmetry as for example explained
in \cite{seidel2008homological}. Orlov introduced the singularity category, generalizing Buchweitz’s categorical
construction to a global setup, and established various important results in \cite{orlov2003triangulated, orlov2009derived, orlov2006triangulated}. It should be noted that the category of matrix factorizations and its stable category are called the category of pairs and the category of D-branes of type $B$, in the terminology of Orlov.  Following the suggestion of Kontsevich, matrix factorizations were used by physicists to describe $D$-branes of type $B$ in Landau-Ginzburg models \cite{kapustin2004d, kapustin2003topological}.
\end{remark}

Inspired by the notion of matrix factorizations of projective modules, we define the category of matrix factorizations of Gorenstein projective modules, as follows.

\begin{dfn}
By a {\it matrix factorization of Gorenstein projective modules}, we mean an ordered pair of $S$-homomorphisms ${\xymatrix{(G_1 \ar@<0.6ex>[r]^{g_1}& G_0\ar@<0.6ex>[l]^{g_0})}}$ such that $G_1,G_0\in \G(S)$ and $g_1g_0=\omega.\id_{G_0}$~~~~,~~~$g_0 g_1=\omega.\id_{G_1}$. A morphism $$\Phi=(\phi_1, \phi_0): {\xymatrix{(G_1 \ar@<0.6ex>[r]^{g_1}& G_0\ar@<0.6ex>[l]^{g_0})}}\lra {\xymatrix{(G_1' \ar@<0.6ex>[r]^{g'_1}& G'_0\ar@<0.6ex>[l]^{g'_0})}}$$ between two matrix factorizations of Gorenstein projective modules is a pair of $S$-homomorphisms $\phi_1: G_1\rt G'_1$ and $\phi_0: G_0\rt G'_0$ such that the diagram
$$\xymatrix{G_1\ar[r]^{g_1} \ar[d]_{\varphi_1} &\ G_0 \ar[d]_{\phi_0}\ar[r]^{g_0} & G_1 \ar[d]^{\phi_1}\\ G'_1 \ar[r]_{g'_1} & G'_0 \ar[r]_{g'_0} & G'_1}
$$commutes.
We denote by $\MF(\omega, \G)$ the category of matrix factorizations of Gorenstein projective modules. It is trivial that $\MF (\omega, G)$ is an additive category, with pointwise direct sums, and $\MF(\omega, \cp)$ is a full subcategory of $\MF(\omega, \G)$.\\
Assume that ${\xymatrix{(G_1 \ar@<0.6ex>[r]^{g_1}& G_0\ar@<0.6ex>[l]^{g_0})}}\in\MF(\omega, \G)$. Since $\omega$ is non-zerodivisor on $G_1$ and $G_0$ and $g_1g_0=\omega.\id_{G_0}$~~~~,~~~$g_0 g_1=\omega.\id_{G_1}$, one may easily infer that $g_1$ and $g_0$ are monomorphisms.
\end{dfn}

\begin{example}Let $k$ be a regular local ring and let $x_1,x_2,x_3,x_4$ be independent variables. Set $S:=k[[x_1,x_2,x_3,x_4]]/{(x^2_1x_3)}$. Evidently, $S$ is a  Gorenstein local ring and  $S/{(x^2)}$ is a Gorenstein projective $S$-modules. Now it is easily seen that ${\xymatrix{(S/{(x^2_1)} \ar@<0.6ex>[r]^{x_2}& S/{(x^2_1)}\ar@<0.6ex>[l]^{x_4})}}\in\MF(x_2x_4, \G)$.
\end{example}

The following interesting result reveals that the category of matrix factorizations of Gorenstein projective modules can be identified with the monomorphism category of Gorenstein projective modules.

\begin{theorem}\label{sg}There is an equivalence functor $F:\mon(\omega, \G)\lrt\MF(\omega, \G)$.
\end{theorem}
\begin{proof}Assume that $(G_1\st{g_1}\rt G_0)\in\mon(\omega, \G)$ is arbitrary. According to Lemma \ref{lem1}, there is a unique morphism $g_0:G_0\rt G_1$ such that ${\xymatrix{(G_1 \ar@<0.6ex>[r]^{g_1}& G_0\ar@<0.6ex>[l]^{g_0})}}\in\MF(\omega, \G)$. Now we set $F(G_1\st{g_1}\rt G_0):={\xymatrix{(G_1 \ar@<0.6ex>[r]^{g_1}& G_0\ar@<0.6ex>[l]^{g_0})}}$. Suppose that $\phi=(\phi_1, \phi_0):(G_1\st{g_1}\rt G_0)\lrt(G'_1\st{g'_1}\rt G'_0)$ is a morphism in $\mon(\omega, \G)$. We would like to show that $\Phi=(\phi_1, \phi_0):{\xymatrix{(G_1 \ar@<0.6ex>[r]^{g_1}& G_0\ar@<0.6ex>[l]^{g_0})}} \lrt {\xymatrix{(G_1' \ar@<0.6ex>[r]^{g'_1}& G'_0\ar@<0.6ex>[l]^{g'_0})}}$
is a morphism in $\MF(\omega, \G)$. To do this, it must be proved that the diagram
\[\xymatrix{ G_1 \ar[r]^{g_1} \ar[d]_{\phi_1} & G_0 \ar[r]^{g_0} \ar[d]_{\phi_0} & G_1 \ar[d]_{\phi_1}
\\ G'_1 \ar[r]^{g'_1} & G'_0 \ar[r]^{g'_0} & G'_1, }\]is commutative.  Since by our hypothesis, the left square is commutative, we only need to show the commutativity of the right one. This indeed follows from the equalities: $g'_1(g'_0\phi_0-\phi_1g_0)=g'_1g'_0\phi_0-g'_1\phi_1g_0=\omega\phi_0-\phi_0g_1g_0=\omega\phi_0-\phi_0\omega=0$, and the fact that $g'_1$ is a monomorphism. So, we define $F(\varphi):= \Phi=(\phi_1, \phi_0)$. Our definition of $F$ reveals that $F$ is an equivalence functor. So the proof is finished.
\end{proof}

As an immediate consequence of Theorem \ref{sg}, we include the result below, which says that $\mon(\omega, \G)$ contains the category of matrix factorizations.
\begin{cor}There is a fully faithful functor $\MF(\omega, \cp)\lrt\mon(\omega, \G)$.
\end{cor}

\begin{remark}\label{11}Assume that
$0\lrt(G'_1\st{g'_1}\rt G'_0)\st{\phi=(\phi_1,\phi)}\lrt (G_1\st{g_1}\rt G_0)\st{\psi=(\psi_1,\psi_0)}\lrt (G''_1\st{g''_1}\rt G''_0)\lrt 0$ is a short exact sequence  in $\mon(\omega, \G)$. As we have seen in the proof of Theorem \ref{sg}, there is a commutative diagram of $S$-modules

$$\xymatrix{0\ar[r]& G'_1\ar[r]^{\phi_1} \ar[d]_{g'_1} & G_1 \ar[d]_{g_1}\ar[r]^{\psi_1}& G''_1 \ar[d]_{g''_1}\ar[r]& 0\\ 0\ar[r]& G'_0 \ar[r]^{\phi_0}\ar[d]_{g'_0} & G_0 \ar[d]_{g_0}\ar[r]^{\psi_0} & G''_0\ar[d]^{g''_0}\ar[r]& 0 \\ 0\ar[r] & G'_1 \ar[r]^{\phi_1} & G_1\ar[r]^{\psi_1}& G''_1\ar[r]& 0,}$$such that the composition of morphisms in each column is equal to multiplication by $\omega$. Moreover, by our hypothesis,
rows are exact. This means that
$$0\lrt {\xymatrix{(G'_1 \ar@<0.6ex>[r]^{g'_1}& G'_0\ar@<0.6ex>[l]^{g'_0})}}\st{\Phi=(\phi_1,\phi_0)}\lra {\xymatrix{(G_1 \ar@<0.6ex>[r]^{g_1}& G_0)\ar@<0.6ex>[l]^{g_0}}}\st{\Psi=(\psi_1,\psi_0)}\lra {\xymatrix{(G''_1 \ar@<0.6ex>[r]^{g''_1}& G''_0)\ar@<0.6ex>[l]^{g''_0}}}\lra 0,$$ is a short exact sequence in $\MF(\omega, \G)$.
Moreover, by the definition of the functor $F$, we have $F(P\st{\omega}\rt P)={\xymatrix{(P \ar@<0.6ex>[r]^{\omega}& P\ar@<0.6ex>[l]^{\id})}}$ and  $F(P\st{\id}\rt P)={\xymatrix{(P \ar@<0.6ex>[r]^{\id}& P\ar@<0.6ex>[l]^{\omega})}}$.
\end{remark}

Assume that $\xi$ is the class of all short exact sequences in $\MF(\omega, \G)$. Combining Theorems \ref{frob} and \ref{sg} and Remark \ref{11}, yields the result below.

\begin{prop}\label{gfrob1} $(\MF(\omega, \G), \xi)$ is an exact category. Moreover,  $\MF(\omega, \G)$  is a Frobenius category and its projective-injective objects are equal to direct summands of finite direct sums of objects of the form ${\xymatrix{(P \ar@<0.6ex>[r]^{\omega}& P\ar@<0.6ex>[l]^{\id})}}\oplus {\xymatrix{(Q \ar@<0.6ex>[r]^{\id}& Q\ar@<0.6ex>[l]^{\omega})}}$, for some projective $S$-modules $P, Q$.
\end{prop}

According to Proposition \ref{gfrob1}, $\MF(\omega, \G)$ is a Frobenius category, and so, its stable category $\underline{\MF}(\omega, \G)$ is triangulated.

\begin{theorem}\label{22}There is a triangle equivalence functor ${\umon(\omega, \G)}\lrt\underline{\MF}(\omega, \G)$.
\end{theorem}
\begin{proof}According to Theorem \ref{sg},  the functor $F:\mon(\omega, \G)\lrt\MF(\omega, \G)$ sending each object $(G_1\st{g_1}\rt G_0)$ to the pair ${\xymatrix{(G_1 \ar@<0.6ex>[r]^{g_1}& G_0\ar@<0.6ex>[l]^{g_0})}}$, is an equivalence of categories. Moreover, as indicated in Remark \ref{11}, $F$ carries projective objects to projectives. So there is an induced functor $\underline{F}:\umon(\omega, \G)\lrt\underline{\MF}(\omega, \G)$, see \cite[ 2.8, page 22]{happel1988triangulated}. Clearly, $\underline{F}$ is an equivalence functor. So, in order to complete the proof, we need to show that $\underline{F}$ is a triangle functor. By Remark \ref{11}, $F$ sends each short exact sequence in $\mon(\omega, \G)$ to a short exact sequence in $\MF(\omega, \G)$. Moreover, $F$ sends projective objects in $\mon(\omega, \G)$ to projective objects in the category $\MF(\omega, \G)$, see Theorem \ref{frob} and Proposition \ref{gfrob1}. These facts enable us to infer that $\underline{F}$ is a triangle functor. So the proof is finished.
\end{proof}

By replacing  Gorenstein projective modules with projectives in the proof of Theorem \ref{22}, one may obtain the following result which recovers \cite[Corollary 3.3]{stablebahlekeh}.
\begin{cor}\label{33}There is a triangle equivalence functor $F':{\umon(\omega, \cp)}\lrt\underline{\MF}(\omega, \cp)$.
\end{cor}

\begin{cor}\label{s2}$\MF(\omega, \cp)$ is a Frobenius subcategory of $\MF(\omega, \G)$, with the same projective objects. In particular, there is a fully faithful triangle functor $\underline{\MF}(\omega, \cp)\lrt\underline{\MF}(\omega, \G)$.
\end{cor}
\begin{proof}As $\cp(S)$ is closed under extensions, one may easily see that $\MF(\omega, \cp)$ is an extension-closed subcategory of $\MF(\omega, \G)$. In view of Proposition \ref{gfrob1}, $\MF(\omega, \G)$ is a Frobenius category. On the other hand, every Gorenstein projective module of finite projective dimension is projective. These facts yield that $\MF(\omega, \cp)$ is an admissible subcategory of $\MF(\omega, \G)$, and so, $\MF(\omega, \cp)$ will be a Frobenius category such that its projective objects are the same as $\MF(\omega, \G)$. This, in turn, implies that the inclusion functor $i:\MF(\omega, \cp)\lrt\MF(\omega, \G)$ induces a fully faithful triangle functor $\underline{i}:\underline{\MF}(\omega, \cp)\lrt\underline{\MF}(\omega, \G)$, as needed.
\end{proof}

In the remainder of this section, we show that the stable category of Gorenstein projective $R$-modules, $\underline{\G}(R)$, can be viewed as a triangulated subcategory of $\umon(\omega, \G)$. First we state an auxiliary lemma, which indicates that any module in $\G(R)$ admits a Gorenstein projective precover over $S$, with projective kernel.

Recall that an $S$-homomorphism $\varphi:G\rt M$ is a Gorenstein precover of $M$, if $G\in\G(S)$,  and for any object $G'\in\G(S)$, the induced map $\Hom_S(G', G)\lrt\Hom_S(G', M)$ is surjective. One should note that since the class of Gorenstein projective modules contains projectives, every Gorenstein projective precover will be surjective.
\begin{lemma}\label{precover}Let $M$ be an arbitrary object of $\G(R)$. Then there is a short exact sequence of $S$-modules $0\rt P\rt G\st{\varphi}\rt M\rt 0$, where $G\in\G(S)$ and $P\in\cp(S)$. In particular, $\varphi: G\rt M$  is a Gorenstein projective precover of $M$.
\end{lemma}
\begin{proof}Since $M\in\G(R)$, its Gorenstein projective dimension over $S$ is one. So, one may take a short exact sequence of $S$-modules $0\rt G\rt P\rt M\rt 0$, with $G\in\G(S)$ and $P\in\cp(S)$. Consider a short exact sequence of $S$-modules $0\rt G\rt Q\rt G'\rt 0$, where $Q\in\cp(S)$ and $G'\in\G(S)$. Hence, applying the functor $\Hom_S(-, P)$ to the latter sequence, gives us a morphism $\theta: Q\rt P$, and so $\varphi: G'\rt M$, making the following commutative diagram of $S$-modules: \[\xymatrix{ 0 \ar[r] & G \ar[r] \ar@{=}[d] & Q \ar[r] \ar[d]_{\theta} & G' \ar[r] \ar[d]_{\varphi} & 0\\  0 \ar[r] & G \ar[r] & P \ar[r] & M \ar[r] & 0.&}\]Without loss of generality, we may assume that $\theta$ is an epimorphism. Thus, $L:=\Ker\theta$ will be a projective $S$-module. Consequently, we get the short exact sequence of $S$-modules, $0\rt L\rt G'\st{\varphi}\rt M\rt 0$. This, in turn, implies that $\varphi: G'\rt M$ is a Gorenstein projective precover of $M$ over $S$. So the proof is completed.
\end{proof}

\begin{theorem}\label{main} There is a fully faithful triangle functor $\underline{T}:\underline{\G}(R)\lrt\umon(\omega, \G)$.
\end{theorem}
\begin{proof}First we define a functor $T:\G(R)\lrt\umon(\omega, \G)$, as follows.
Assume that $M$ is an arbitrary object of $\G(R)$. In view of Lemma \ref{precover}, there exists a short exact sequence of $S$-modules $0\rt P\st{g}\rt G\st{\varphi}\rt M\rt 0$, where $P\in\cp(S)$ and $\varphi:G\rt M$ is a Gorenstein projective precover of $M$. So $(P\st{g}\rt G)\in\mon(\omega, \G)$ and we define $T(M):=(P\st{g}\rt G)$.  Next, for a given morphism $f:M\rt M'$ in $\G(R)$, one may obtain the following commutative diagram of $S$-modules with exact rows:\[\xymatrix{0 \ar[r] & P \ar[r]^{g} \ar[d]_{\varphi_1} & G \ar[r] \ar[d]_{\varphi_0} & M \ar[r] \ar[d]_{f} & 0\\  0 \ar[r] & P' \ar[r]^{g'} & G' \ar[r] & M' \ar[r] & 0,&}\] in which $P, P'\in\cp(S)$. It should be noted that the existence of $\varphi_0$ comes from the facts that $G'\rt M'$ is a Gorenstein projective precover and $G\in\G(S)$. In particular, we get the morphism $\varphi:=(\varphi_1, \varphi_0): (P\st{g}\rt G)\lrt (P'\st{g'}\rt G')$ in $\umon(\omega, \G)$, and we set $T(f):=(\varphi_1, \varphi_0)$.  The first issue to address is the well-defindness of $T$. It is obvious that for given two morphisms $f,f':M\rt M'$ in $\G(R)$, $T(f+f')=T(f)+T(f')$. Now suppose that $f: M\rt M'$ is a zero morphism. Keeping the notation of the latter diagram, one may find a morphism $t: G\rt P'$ such that $g't=\varphi_0$. Moreover, as $g'$ is a monomorphism, one may deduce that $tg=\varphi_1$. In particular, we obtain the following commutative diagram of $S$-modules:
{\tiny{\[\xymatrix@C-0.6pc@R-1.4pc{ && P \ar[rr]^{g} \ar[ddd]_{\varphi_1}
\ar[ddr]^{\varphi_1}  && G  \ar[ddd]^{\varphi_0} \ar[ddr]^{t} \\ \\
& && P' \ar[rr]^{\id} \ar[dl]_{\id} &&  P'\ar[dl]^{g'}      \\
&& P' \ar[rr]^{g'} && G'
}\]}}Consequently, the morphism $\varphi=(\varphi_1,\varphi_0)$ factors through the projective object $(P'\st{\id}\rt P')$ in $\mon(\omega, \G)$, and so, $T(f)=(\varphi_1,\varphi_0)$ is zero in $\umon(\omega, \G)$. Thus $T$ is well-defined. Next assume that a morphism $f: M\rt M'$ in $\G(R)$ factors through a projective object $P$. As $R$ is local, $P=\oplus_{i=1}^{n}R$. Considering the short exact sequence of $S$-modules $0\rt S\st{\omega}\rt S\rt R\rt 0$, one may easily infer that $T(f)=(\varphi_1, \varphi_0)$ factors through the projective object $\oplus_{i=1}^{n}(S\st{\omega}\rt S)$ in $\mon(\omega, \G)$. So, there is an induced functor $\underline{T}:\underline{\G}(R)\lrt\umon(\omega, \G)$. Evidently, the functor $\underline{T}$ is full. Now we show that $\underline{T}$ is faithful. To see this, assume that $f:M\rt M'$ is a morphism in $\G(R)$ such that $\underline{T}(f)=(\varphi_1, \varphi_0)=0$ in $\umon(\omega, \G)$. It must be proved that $f=0$ in $\underline{\G}(R)$. By our hypothesis, $(\varphi_1, \varphi_0)$ factors through a projective object of $\mon(\omega, \G)$. In view of Theorem \ref{frob}, we may assume that $(\varphi_1, \varphi_0)$ factors through of finite direct sum $\oplus_{i=1}^n(S\oplus S\st{\omega\oplus\id}\lrt S\oplus S)$. Therefore, one has the following commutative diagram of $S$-modules:
$$\xymatrix{0\ar[r]& P\ar[r] \ar[d]_{h_1} & G' \ar[d]_{h_0}\ar[r]& M \ar[d]_{l}\ar[r]& 0\\ 0\ar[r]& \oplus_{i=1}^n(S\oplus S) \ar[r]^{\oplus(\omega\oplus \id)}\ar[d]_{h'_1} & \oplus_{i=1}^n(S\oplus S) \ar[d]_{h'_0}\ar[r] & \oplus_{i=1}^nR\ar[d]_{l'}\ar[r]& 0 \\ 0\ar[r] & P' \ar[r] & G'\ar[r]& M'\ar[r]& 0,}$$where $h'_1h_1=\varphi_1$, $h'_0h_0=\varphi_0$ and { $l, l'$ are induced morphisms. In particular, the commutativity of the diagram yields that $f=l'l$. Thus $f$} factors through a projective object, and so, $f=0$ in $\underline{\G}(R)$, as desired. Hence, it remains to show that $\underline{T}$ is a triangle functor. {To do this, we show that the image of any short exact sequence in $\G(R)$ comes from a short exact sequence in $\mon(\omega, \G)$. Namely, assume that $0\rt M'\st{f}\rt M\st{g}\rt M''\rt 0$  is a short exact sequence in $\G(R)$. We prove that there exist a short exact sequence $0\lrt (G'_1\st{g'}\rt G'_0)\lrt (G_1\st{g}\rt G_0)\lrt (G''_1\st{g''}\rt G''_0)\lrt 0$ in $\mon(\omega, \G)$, such that the sequence $T(M')\st{T(f)}\lrt T(M)\st{T(g)}\lrt T(M'')$ is nothing more than the sequence $ (G'_1\st{g'}\rt G'_0)\lrt (G_1\st{g}\rt G_0)\lrt (G''_1\st{g''}\rt G''_0)$ in $\umon(\omega, \G)$.
 To do this,  Consider short exact sequences of $S$-modules, $0\rt P\st{l}\rt G\st{\pi}\rt M\rt 0$ and $0\rt P''\st{l''}\rt G''\st{\pi''}\rt M''\rt 0$, which are obtained by using Lemma \ref{precover}. Since $G''\st{\pi''}\rt M''$ is a Gorenstein projective precover and $G\in\G(S)$, there exists an $S$-homomorphism $h: G\rt G''$ such that $\pi''h=g\pi$. So, one may get a commutative diagram of $S$-modules with exact rows: \[\xymatrix{0 \ar[r] & P''\oplus P \ar[r]^{\id\oplus l} \ar[d]_{[\id~~t]} & P''\oplus G \ar[r]^{[0~~\pi]} \ar[d]_{[l''~~h]} & M \ar[r] \ar[d]_{g} & 0\\  0 \ar[r] & P'' \ar[r]^{l''} & G'' \ar[r]^{\pi''} & M'' \ar[r] & 0,&}\]where vertical maps are surjective, and $t$ is an induced morphism. This, in particular, gives us the following commutative diagram of $S$-modules with exact rows and columns:
\[\xymatrix@C-0.5pc@R-.8pc{&0\ar[d]&0\ar[d]&0 \ar[d]  && \\ 0\ar[r]&
P'~\ar[r]\ar[d]& P''\oplus P\ar[r]\ar[d]& P''\ar[d]\ar[r] & 0 \\ 0 \ar[r]  & G'~\ar[r]\ar[d]_{\pi'} & P''\oplus G\ar[r]\ar[d]_{[0~~\pi]}& G''\ar[d]_{\pi''} \ar[r]& 0&\\
0 \ar[r]  & M'~\ar[r]^{f}\ar[d] & M\ar[r]^{g}\ar[d]& M''\ar[d] \ar[r]& 0&\\ &0&0& 0&&&\\ }\]where $P'$ and $G'$ are projective and Gorenstein projective, respectively. One should note that, since the map $\Ker [0~~\pi]\lrt\Ker\pi''$ is surjective, the snake lemma guarantees that the same is true for $\pi'$. Thus, according to the left column,  $G'\st{\pi'}\rt M'$ is a Gorenstein projective precover of $M'$ with a projective kernel. Consequently, the above diagram gives us the short exact sequence $0\lrt (P'\rt G')\lrt (P''\oplus P\rt P''\oplus G)\lrt (P''\rt G'')\lrt 0$ in $\mon(\omega, \G)$. In particular, by the definition of the functor $T$, the sequence $T(M')\st{T(f)}\lrt T(M)\st{T(g)}\lrt T(M'')$ in $\umon(\omega, \G)$, is indeed the sequence $ (P'\rt G')\lrt (P''\oplus P\rt P''\oplus G)\lrt (P''\rt G'')$.  Hence the proof is completed.}
\end{proof}

According to Theorem \ref{22}, $\umon(\omega, \G)$ is triangle equivalent to the stable category $\underline{\MF}(\omega, \G)$. This, combined with Theorem \ref{main}, yields the following result.
\begin{cor}There is a fully faithful triangle functor $\underline{\G}(R)\lrt\underline{\MF}(\omega, \G)$.
\end{cor}

It is worthwhile to compare the above result with \cite[Theorem 3.5]{bergh2015complete}, where it was shown that there is a fully faithful triangle functor from the homotopy category of matrix factorizations, which is the same as the stable category $\underline{\MF}(\omega, \cp)$, to the stable category $\underline{\G}(R)$.

{The singularity category $\ds(R)$ is defined to be the Verdier quotient of the bounded derived category $\db(R)$ of $R$ modulo the subcategory of perfect complexes. This notion is traced back to   Buchweitz \cite{buchweitz1987maximal} and is rediscovered in the geometric context by Orlov \cite{orlov2003triangulated}.}
By a fundamental result of Buchweitz and Happel \cite[Theorem 4.4.1]{buchweitz1987maximal} and \cite[Theorem 4.6]{happel1991gorenstein}, the stable category $\underline{\G}(R)$ is triangle equivalent to the singularity category $\ds(R)$ of $R$, provided that $R$ is Gorenstein. So, as an immediate consequence of Theorem \ref{main}, we include the result below.
\begin{cor}\label{dd}Let $R$ be a Gorenstein ring. Then there is a fully faithful triangle functor $F:\ds(R)\lrt\umon(\omega, \G)$.
\end{cor}


\bibliographystyle{siam}

\end{document}